\documentclass[12pt]{article}

\usepackage{amssymb}
\usepackage{lineno}
\usepackage{color}
\usepackage{graphicx}
\usepackage{amsmath}
\usepackage{enumerate}

\begin{document}
\newtheorem{Theorem}{Theorem}[section]
\newtheorem{Proposition}[Theorem]{Proposition}
\newtheorem{Lemma}[Theorem]{Lemma}
\newtheorem{Example}[Theorem]{Example}
\newtheorem{Corollary}[Theorem]{Corollary}
\newtheorem{Fact}[Theorem]{Fact}
\newtheorem{Conjecture}[Theorem]{Conjecture}
\newenvironment{Definition} {\refstepcounter{Theorem} \medskip\noindent
 {\bf Definition \arabic{section}.\arabic{Theorem}}\ }{\hfill}
\newenvironment{Remarks} {\refstepcounter{Theorem}
\medskip\noindent {\bf Remarks
\arabic{section}.\arabic{Theorem}}\ }{\hfill}

\newenvironment{Question} {\medskip\refstepcounter{Theorem}
     \noindent {\bf Question
\arabic{section}.\arabic{Theorem}}\ } {\hfill}

 \newcommand{\qed}{\hfill \ensuremath{\Box}}
\newenvironment{Proof}{{\noindent \bf Proof\ }}{\hfill\qed}

\newenvironment{claim} {{\smallskip\noindent \bf Claim\ }}{\hfill}

\def \blue {\color{blue}}
\def\red{\color{red}}
\def \L {{\cal L}}

\def \B {{\cal B}}
\def \mod {{\rm mod \ }}
\def \iso {\cong}
\def \Lor  {\L_{\rm or}}
\def \Lr {\L_{\rm r}}
\def \Lg {\L_{\rm g}}
\def \I {{\cal I}}
\def \M {{\cal M}}\def\N {{\cal N}}
\def \E {{\cal E}}
\def \Proj {{\mathbb P}}
\def \H {{\cal H}}
\def \x {\times}
\def \Stab {{\rm Stab}}
\def \Z {{\mathbb Z}}
\def \V {{\mathbb V}}
\def \C {{\mathbb C}} \def \Cexp {\C_{\rm exp}}
\def \R {{\mathbb R}}
\def \Q {{\mathbb Q}}\def \K {{\mathbb K}}
\def \F {{\cal F}}\def \A {{\mathbb A}}
 \def \X {{\mathbb X}}
\def \G {{\mathbb G}}
\def\HH {{\mathbb H}}
\def \Nn {{\mathbb N}}\def \Nn {{\mathbb N}}
\def\D {{\mathbb D}}
\def \hat {\widehat}
\def \bar{\overline}
\def \Spec {{\rm Spec}}
\def \bul {$\bullet$\ }
\def\proves {\vdash}
\def \Co {{\cal C}}
\def \ACFp {{\rm ACF}_p}
\def \ACF0 {{\rm ACF}_0}
\def \ee {\prec}
\def \Diag {{\rm Diag}}
\def \Diage {{\rm Diag}_{\rm el}}
\def \DLO {{\rm DLO}}
\def \d {{\rm depth}}
 \def \dist {{\rm dist}}
\def \P {{\cal P}}
\def \ds {\displaystyle}
\def \Fp {{\mathbb F}_p}
\def \acl {{\rm acl}}
\def \dcl {{\rm dcl}}

\def \dom {{\rm dom}}
\def \tp {{\rm tp}}
\def \stp {{\rm stp}}
\def \Th  {{\rm Th}}
\def\< {\Lngle}
\def \> {\rangle}
\def \n {\noindent}
\def \minusdot{\hbox{\ {$-$} \kern -.86em\raise .2em \hbox{$\cdot \
$}}}
\def\exp {{\rm exp}}\def\ex {{\rm ex}}
\def \td {{\rm td}\ }
\def \ld {{\rm ld}}
\def \span {{\rm span}}
\def \tilde {\widetilde}
\def \d {\partial}
\def \del {\partial}
\def \cl {{\rm cl}}
\def \acl {{\rm acl}}
\def \cN {{\cal N}}
\def \Qalg {{\Q^{\rm alg}}}
\def \th {^{\rm th}}
\def \deg { {\rm deg} }
\def\hat {\widehat}
\def\li {\L_{\infty,\omega}}
\def\lo {\L_{\omega_1,\omega}}
\def\lk {\L_{\kappa,\omega}}
\def \ee {\prec}\def \bSigma {{\mathbf\Sigma}}
\def \mod {\ {\rm mod\ }}
\def \Tor {{\rm Tor}}
\def \td {{\rm td}}
\def \dim {{\rm dim\ }}
 \def \| {\kern -.3em \restriction \kern -.3em}
 \def \lc{\lceil}
 \def \rc{\rceil}
 \def \SR {{\rm SR}}
 \def\L {{\cal L}}
 \def \a{{\bf a}}
 \def \b{{\bf b}}
 \def\x{{\bf x}}
 \def \< { \langle}
 \def \> { \rangle}
 \def \K {{\cal K}}
  \def \MM {{\mathbb M}} 
  \def \rcl {{\rm rcl}}
  \def\acl{{\rm acl}}
  \def\img{{\rm img}}
  \def \bfalpha {\mbox{\boldmath$\alpha$}}
  
   \title{Finding suitably generic points on curves with an application to the construction of rigid real closed fields}
 \author{Dragos Ghioca\\ University of British Columbia \and David Marker\\ University of Illinois Chicago\and Charles Steinhorn\\ Vassar College}
 \date{}
\maketitle
  
  \begin{abstract} Let $K$ be an algebraically closed field of characteristic 0 and transcendence degree at least 2.  Let $C\subset K^2$ be an irreducible curve defined over $K$ but not defined over the algebraic closure of $ \Q$.  There is $(x ,y)$ a $K$-point of $C$ such that $x$ and $y$ are algebraically independent.  Moreover, if $C_0$ and $C_1$ are two such curves and there is a finite-to-finite algebraic correspondence between them defined over $K$, then there are corresponding $K$-points $(x_0,y_0)\in C_0$ and $(x_1,y_1)\in C_1$ such that
  $x_0$ and $y_0$ are algebraically independent and $x_1$ and $y_1$ are algebraically independent.  We use the latter result to construct non-Archimedean real closed fields of transcendence degree $\kappa$ with no non-trivial automorphisms for all $2\le\kappa\le \aleph_1$.
  \end{abstract} 
\section{Introduction}
 
Suppose $C\subset\C^n$ is an irreducible curve that is not contained in a hypersurface defined over $\bar Q$, the algebraic closure  of $\Q$.  If $(x_1,\dots,x_n)$ is a generic point of $C$, then $x_1,\dots,x_n$ are algebraically independent. Indeed $$\{(x_1,\dots,x_n)\in C: x_1,\dots, x_n\hbox{ are algebraically dependent}\}.$$ is a countable subset of $C$. More generally, if $k$ is the field of definition of $C$, $K\supset k$ is algebraically closed and the transcendence degree of $K$ over $k$ is at least $n$ we can find a $K$-point $(x_1,\dots,x_n)\in C$ where $x_1,\dots,x_n$ are algebraically independent.  But what happens when the transcendence degree of $K/k$ is small? For example, what happens if $K$ is the algebraic closure of $k$? The concept of heights have been effective in addressing problems in arithmetic geometry---see the introduction to \cite{cgmm}, e.g.---and we employ heights here to address these questions. 

We first consider the case where $K$ is an algebraically closed field of characteristic 0 whose transcendence degree is at least $2$ and $C\subset K^2$ is an irreducible curve defined over $K$ that is not contained in a curve defined over $\bar \Q$.  In Proposition \ref{basic}, we show that $C$ always contains a $K$-point $(x,y)$ with $x$ and $y$ algebraically independent. 
Indeed, the height machinery allows us to make precise the intuition that ``most" $K$-points of $C$ have algebraically independent coordinates.
More generally, suppose $C_0$ and $C_1$ are two such curves in  $K^2$ for which there is a finite-to-finite correspondence between them defined over $K$.  It follows from Theorem \ref{main} that there are corresponding $K$-points $(x_0,y_0)\in C_0$ and $(x_1,y_1)\in C_1$ such that both pairs $(x_0, y_0)$ and $(x_1, y_1)$ are algebraically independent. 

We prove Theorem \ref{main} by applying a result of Ghioca, Masser and Zannier \cite{gmz} on bounding heights. There they prove Conjecture~1.6 of \cite{cgmm} for all plane curves defined over a field of characteristic~0, a special case of the Bounded Height Conjecture~1.8  formulated in \cite{cgmm}.  Proposition \ref{basic} has more elementary proofs--and is perhaps well known--but we provide a proof using \cite{gmz} to illustrate the method first in a simpler case.
The second and third authors conjectured Theorem~\ref {main} to complete a construction of rigid non-Archimedean real closed fields--see \S \ref{rigid}. They were able to prove Proposition \ref{basic} and some cases of Theorem \ref{main}, but needed the help of the first author and the tools of \cite{gmz} to complete the proof. 

The paper is organized as follows. In \S\ref{heights} we review some properties of Weil heights that we need to prove Theorem \ref{main}, including Theorem~1.4 of Ghioca, Masser and Zannier \cite{gmz}. The proofs of Proposition \ref{basic} and Theorem \ref{main} are contained in \S \ref{MT}. The context changes in \S\ref{real} to curves defined over real closed fields. Finally, in \S \ref{rigid} we apply Theorem \ref{main} to give a new construction of rigid non-Archimedean real closed fields of transcendence degree $\kappa$ for $\kappa\le \aleph_1.$ Throughout this paper, the algebraic closure of a field $K$ will be denoted by $\bar K$,  the transcendence degree of $K$ will be denoted by $\td(K)$, and if 
$k\subset K$, $\td(K/k)$ will represent the transcendence degree of $K/k$.  

\medskip
We  thank Jim Freitag and Tom Scanlon for helpful conversations.

 \section{Preliminaries on heights}\label{heights}
  
 We follow the treatment of Weil heights from \cite{gmz}.
Let $k\subset K$ be algebraically closed fields where $1\le \td(K/k)<\infty$.  Let $\bfalpha$ be a transcendence base for $K/k$.
Let $x\in K\setminus k(\bfalpha)$.  Let $d$ be the degree of $k(\bfalpha,x)/k(\bfalpha)$.  Let $p_0,\dots,p_d$ be relatively prime polynomials over $k$ such that $$\sum p_i(\bfalpha) x^i=0.$$ The Weil height of $x$ is given by
$$h(x)= {\max_i {\rm deg}(p_i)\over d}.$$
If $(x_1,\dots,x_n)\in K^n$ we define $h((x_1,\dots,x_n))=\sum h(x_i)$.

The following theorem will be our main tool.
  
  \begin{Theorem} [Ghioca--Masser--Zannier \cite{gmz}]\label{GMZ} Let $k\subset K$ be algebraically closed fields of characteristic $0$ with $2\le \td(K/k)<\infty$. Let $C\subset K^2$ be an irreducible algebraic curve defined over $K$ which is not defined over a subfield of transcendence degree at most $1$.  Let $X\subset K^2$ be all points of $K^2$ contained in a curve defined over $k$. Then there is $B$ such that all points of $X\cap C$ have height at most $B$.
  \end{Theorem}
  
  We also need the following Lemma which follows, for example, from Lemma 3.3 (b) of \cite{gn}.

 \begin{Lemma}\label{compare} Let $k\subset K$ be algebraically closed fields with $1\le \td(K/k)<\infty$.  Let $C\subset 
 K^n$ be an irreducible curve  defined over $K$ that projects dominantly onto each coordinate.  For any $1\le i,j\le n$ there are constants $A$ and $B$
 such that if $(x_1,\dots x_n)\in C$ then $$h(x_i)<A h(x_j)+B.$$ 
 \end{Lemma}

 In particular, for any $B_0$ we can find a $B$ such that if $(x_1,\dots,x_n)\in C$ and $x_i$ has height at least $B$, then $x_j$ has height at least $B_0$.
 
\medskip In \S \ref{real} we will need the following basic facts about heights in real closed fields.
 
\begin{Lemma} \label{real heights} Let $k\subset K$ be real closed fields with $1\le \td(K/k)<\infty.$ Let $\bfalpha$ be a transcendence base for $K/k$ and consider the corresponding Weil height.  If $a,b\in K$ with $a<b$, the interval $(a,b)$
  contains points of arbitrarily large height.  In particular, there is $c\in (a,b)$ transcendental over $k$.
\end{Lemma}
\begin{Proof} First note that $K$ contains elements of arbitrarily large height. For example, if $\alpha$ is a transcendence base and 
$p(${\bf X}$)\in k[ ${\bf X}$]$ has degree $d$, 
 then $\sqrt [3]{p(\bfalpha)}$ has height $d/3$.

Using fractional linear transformations we can map any interval to any other interval.  In particular, we can map an element of arbitrarily large height into any particular interval $(a,b)$.  Lemma \ref{compare} then tells us $(a,b)$ contains elements of arbitrarily large height.
\end{Proof}
 
  \section{Main theorem}\label{MT}
  
  \begin{Proposition}\label{basic} Let $K$ be an algebraically closed field of characteristic $0$ with $2\le\td(K)$.  Let $C\subset K^2$ be an irreducible curve defined over $K$ but
  not defined over $\bar{\Q}$. There is $(x,y)\in C$ such that $x$ and $y$ are algebraically independent.
  \end{Proposition}
  \begin{Proof} Let $k$ be the algebraic closure of the field of definition of $C$.  There are two cases to consider.
  
  \smallskip \n {\bf case 1} $\td(k)<\td(K)$.
  
  If $C^\prime$ is a curve defined over $\bar \Q$, then $C\cap C^\prime$ is finite and hence a subset of $k^2$.  Thus if $(x,y)\in C$ and $x\not \in k$,
  then $x$ and $y$ are algebraically independent.
  
  \smallskip \n {\bf case 2} $\td(k)=\td(K)$.
  
 In this case $K$ must have finite transcendence degree.  Let $\bfalpha$ be a transcendence base for $K$ and we consider the Weil heights for points of $K$ over $\bar{\Q}({\bfalpha})$.
By Theorem \ref{GMZ}, there is a bound on the heights of points of $C$ that are also on curves defined over $\bar \Q$.  In particular,
there is $(x,y)\in C$ not contained in any curve defined over $\bar\Q$.  Then $x$ and $y$ are algebraically independent. 
  \end{Proof}
 
  \begin{Theorem}\label{main} Let $K$ be an algebraically closed field of characteristic $0$ with $\td(K)\ge 2$.  Let $C\subset K^4$ be an irreducible curve.  Let $C_0, C_1\subset K^2$ be the Zariski closures of the  projections of $C$ to the first two, respectively last two, coordinates.  Suppose that neither $C_0$ nor $C_1$ is  contained in a curve defined over $\bar\Q$.  Then there is $(x_0,y_0,x_1,y_1)\in C$ such that $x_0$ and $y_0$ are algebraically independent and $x_1$ and $y_1$ are algebraically independent.
  \end{Theorem}
  
 Let $K$, $C$, $C_0$ and $C_1$ be as in the statement of the theorem.
 If $K$ has infinite transcendence degree, we could replace $K$ by any algebraically closed subfield of transcendence degree at least 2 over which $C$ is defined.  Thus we may, without loss of generality, assume that $K$ has finite transcendence degree.\footnote{More to the point, if $K$ has infinite transcendence degree we can choose $(x_0,y_0,x_1,y_1)$ such that some coordinate
 is transcendental over the field of definition of $C$ and this point will have the desired property.}
 
 We prove Theorem \ref{main} in a sequence of lemmas.  We first deal with the most general case.
 
 \begin{Lemma}\label{gen case} Suppose $C$ projects dominantly onto each coordinate and neither $C_0$ nor $C_1$ is defined over a subfield of transcendence degree $1$. Then there is $(x_0,y_0,x_1,y_1)\in C$ with $x_0,y_0$ algebraically independent and $x_1,y_1$ algebraically independent.
 \end{Lemma}
 \begin{Proof} Let $\bfalpha$ be a transcendence base for $K$ and consider the associated Weil height. By Theorem \ref{GMZ}, there are $B_0$ and $B_1$, such if $(x,y)\in C_i$ and $h(x_i)\ge B_i$, then $x$ and $y$ are algebraically independent.
 By Lemma \ref {compare}, there is $B_0^\prime \ge B_0$ such that if $(x_0,y_0,x_1,y_1)\in C$ and $h(x_0)\ge B_0^\prime$,
 then $h(x_1) \ge B_1$.  In this case $x_0$ and $y_0$ are algebraically independent and $x_1$ and $y_1$ are algebraically independent.
 \end{Proof}
 
 \medskip The next  lemmas deal with cases where some of the assumptions of Lemma \ref{gen case} fail.

\begin{Lemma} Suppose either $C_0$ or $C_1$ is a point. Then there is $(x_0,y_0,x_1,y_1)\in C$ such that
$x_0,y_0$ are algebraically independent and $x_1,y_1$ are algebraically independent.
 \end{Lemma}
 \begin{Proof}
 Say $C_0$ is the point $(a,b)$. Because $C_0$ is not contained in a curve defined over $\bar \Q$, $a$ and $b$ are algebraically independent.  By Proposition \ref{basic}, we can find $(x_1,y_1)\in C_1$ algebraically independent. Then $(a,b,x_1,y_1)$ is our desired point of $C$.
 \end{Proof}
 
 \medskip
 Thus, we may, without loss of generality, assume that $C_0$ and $C_1$ are both curves.

  \begin{Lemma}\label{both td 1} Suppose both $C_0$ and $C_1$ are defined over subfields of transcendence degree $1$. Then we can find $(x_0,y_0,x_1,y_1)\in C$ with $x_0,y_0$ algebraically independent and  $x_1,y_1$ algebraically independent.
  \end{Lemma}
  \begin{Proof}  Let $C_0$ be defined over $\bar{\Q(a)}$.  
  
  Permuting coordinates if necessary, we may assume that the projections of $C$ onto the first and third coordinates are dominant.

  \smallskip \n {\bf case 1} $C_1$ is defined over $\bar{\Q(a)}$.
  
  Let $\bfalpha$ be a transcendence base for $K$ over $\bar{\Q(a)}$ and  consider the Weil heights over  the base field $\bar{\Q(a)}(\bfalpha)$.   By Lemma \ref{compare} we can choose $(x_0,y_0,x_1,y_1)\in C$ where $x_0$ is of sufficiently large height that $x_1$ has positive height.  In particular, both $x_0$ and $x_1$ are transcendental over $\Q(a)$.   If $(u,v)\in C_i$ and $u$ and $v$ are algebraically dependent, then $u,v\in \bar{\Q(a)}$. Thus $x_0$ and $y_0$ are algebraically independent and $x_1$ and $y_1$ are algebraically independent.
  
  \smallskip \n {\bf case 2} $C_1$ is not defined over $\bar{\Q(a)}$.
  
  Suppose $C_1$ is defined over $\bar{\Q(b)}$ where $a$ and $b$ are algebraically independent.
  Let $a,b,\bfalpha$ be a transcendence base for $K$ and consider the corresponding Weil height.
  Let $\Gamma\subset K^2$ be the projection of $C$ onto the first and third coordinates. We first 
  consider the case where $\Gamma$ is not defined over a field of transcendence degree at most 1.
 By Theorem \ref{GMZ}, we can  choose $(x_0,y_0,x_1,y_1)\in C$ where $x_1\in \bar{\Q(a)}$ 
  has sufficiently large height  to insure $x_0$ and $x_1$ are algebraically independent.
 Since $x_1\not\in \bar{\Q(b)}$, $x_1$ and $y_1$ are algebraically independent.
 Since $x_1$ and $a$ are interalgebraic over $\Q$, $x_0$ is algebraically independent from $a$.
 Thus $x_0$ and $y_0$ are algebraically independent.
 
 Next, suppose that $\Gamma$ is defined over a field of transcendence degree at most one. 
 First we consider the case where  $\Gamma$ is defined over $\bar{\Q(a)}$. Let $a,\bar\alpha$ be a transcendence base for
 $K$ over $\bar {\Q(b)}$ and consider the corresponding Weil height. 
 Choose $(x_0,y_0,x_1,y_1)\in C$
 where $x_0\in \bar{\Q(a+b)}$ has sufficiently large height. In particular $x_0$ is transcendental over $\Q(a)$, so $x_0$ and $y_0$ are algebraically independent.
 By Lemma \ref{compare}, we can choose $x_0$ of sufficiently large height that $x_1$ is transcendental over $\Q(b)$.
 Then $x_1$ and $y_1$ are algebraically independent.  
  
The case where $\Gamma$ is defined over $\bar {\Q(b)}$ is similar. Finally, consider the case where $\Gamma$ is defined over $\bar{\Q(c)}$ where $a,b,c$ are algebraically independent.  By Lemma \ref{compare}, we can choose $x_0\in\bar{\Q(c)}$ of sufficiently large height to insure $x_1$ is transcendental.  Since $\Gamma$ is defined over $\bar{\Q(c)}$, $x_1\in \bar{\Q(c)}$. Since $x_0$ is transcendental over $\Q(a)$, $x_0$ and $y_0$ are algebraically independent. Similarly, $x_1$ and $y_1$ are algebraically
independent. 
  \end{Proof}
  
  \begin{Lemma} If some coordinate projection is constant, then there is $(x_0,y_0,x_1,y_1)\in C$ such that
  $x_0,y_0$ are algebraically independent and $x_1,y_1$ are algebraically independent.
  \end{Lemma}
 \begin{Proof} Without loss of generality assume $C\subset \{a\}\times K^3$. By our assumptions on $C_0$, we must have $a\not\in \bar\Q$.
Let $a,\bfalpha$ be a transcendence base for $K$ and we work with the associated Weil height.  Suppose $C_1$ is not defined over a subfield of transcendence degree 1.  As in the proof of Propositon \ref{basic}, there is a bound $B$ such that if 
$(x_1,y_1)\in C_1$ has height at least $B$, then $x_1$ and $y_1$ are algebraically independent.  By Lemma \ref{compare}, we can choose $y_0$ algebraically independent from $a$ such that  $(a,y_0,x_1,y_1)\in C$ and $h(x_1)\ge B$. Then $x_0$ and $y_0$ 
are algebraically independent and $x_1$ and $y_1$ are algebraically independent.

If $C_1$ is defined over a subfield of transcendence degree one, then we are done by Lemma \ref{both td 1}.
\end{Proof}

\medskip
Thus we may assume that all coordinate projection maps from $C$ are dominant.

\begin {Lemma}\label{last lemma} Suppose $C_0$ is defined over a subfield of transcendence degree $1$ and $C_1$ is not.  Then
there is $(x_0,y_0,x_1,y_1)\in C$ such that $x_0$ and $y_0$ are algebraically independent and $x_1$ and $y_1$ are algebraically independent.
\end{Lemma}
 \begin{Proof} Suppose $C_0$ is defined over $\bar{\Q(a)}$ and $C_1$ is defined over $\bar{\Q(a,b_1,\dots,b_m)}$ but 
 not over $\bar{\Q(a)}$, where $b_1,\dots,b_m$ are algebraically independent over $\Q(a)$ and $m\ge 1$.  By Theorem \ref{GMZ},
 there is $B$ such that if $(x_1,y_1)\in C_1$ and $h(x_1)\ge B$, then $x_1$ and $y_1$ are algebraically independent.
 By Lemma \ref{compare}, we can choose $(x_0,y_0,x_1,y_1)\in C$ where $x_0\in\bar{\Q(b_1)}$ is of large enough height  to insure
 that $x_1$ has height at least $B$ and $x_1$ and $y_1$ are algebraically independent.  Since $x_0$ is transcendental over $\bar{\Q(a)}$, $x_0$ and $y_0$ are algebraically independent.
 \end{Proof}

 \medskip  Taken together Lemmas \ref{gen case}--\ref{last lemma} yield the proof of Theorem \ref{main}.
 
 \medskip We expect that Theorem \ref{main} could be generalized to show that if $K$ is an algebraically closed field of characteristic zero and $C\subset K^{2n}$ is an irreducible curve such that no $C_i$ is  contained in a curve defined over $\bar \Q$, where $C_i\subset K^2$ is the projection onto coordinates $2i+1$ and $2i+2$, then there is 
 $(x_0,y_0,\dots, x_{n-1},y_{n-1})\in C$
 such that $x_i$ and $y_i$ are algebraically independent for all $i$.  The generic case where each coordinate projection map from $C$ is dominant and no $C_i$ is  defined over a subfield of transcendence degree~1 follows exactly as in Lemma \ref{gen case}, but work is needed to hash out the details when one of these additional assumptions fails.

  \section{Real  algebraic consequences}\label{real}
  
  Let $K$ be a real closed field and let $(a,b)$ be an interval in $K$.
  
  \begin{Definition} Suppose $f,g:(a,b)\rightarrow K$ are definable over $K$.  We say that $f$ and $g$ are {\em algebraically independent}  
  if there is no nonzero polynomial $p(X,Y)$ defined over $\Q^\rcl$ such that $p(f(x),g(x))=0$ on some subinterval of $(a,b)$.
  \end{Definition}
  
  \begin{Corollary} \label{real main} If $K$ is a real closed field of transcendence degree at least 2 and $f,g:(a,b)\rightarrow K$ are definable over $K$ and  algebraically 
  independent, then there is $z\in (a,b)$ such that $f(z),g(z)$ are algebraically independent.
  \end{Corollary}
  \begin{Proof}  We may assume $f$ and $g$ are strictly monotonic on $(a,b)$.  Suppose $f$ is constant, say $f(x)=d$ 
  for $x\in (a,b)$. We must have $d$ transcendental, as otherwise the graph of $(f,g)$ is contained in the curve $X=d$.
  If $g$ is non-constant we can find $z\in(a,b)$ such that $g(z)$ algebraically independent from $d$.  If $g$ is also constant,
  say $g(x)=e$, then algebraic independence tells us $d$ and $e$ must be algebraically independent, so any $z\in (a,b)$ suffices.
  Thus, without loss of generality, we may assume $f$ is non-constant.  Similarly, we may assume that $g$ is non-constant.
  
  By quantifier elimination, there is a curve $C_0$ defined over $K$ such that $(f(x),g(x))\in C_0$ for all $(x,y)\in (a,b)$.
   We now move to $\bar K$.  Let $C$ be an irreducible  curve over $\bar K$ contained in the Zariski closure of $C_0(K)$
   such that $C\cap \img(f,g)$ is infinite.  Following the proof of Proposition \ref{basic}, if $C$ is defined over a subfield $k$ with 
   $\td(k)<\td(K)$, then it suffices to find $x\in\img(f)$ with $x\not\in k$ and if $\td(k)=\td(K)$ it suffices to find $(x,y)$ of sufficiently large height.  In either case the existence of the desired $x$ follows from Lemma \ref{real heights}.
  \end{Proof}
  
  \medskip A similar analysis of the proof of Theorem \ref {main} leads to the following Corollary that will be used in Section 
  \ref{rigid}.
  
  \begin{Corollary}\label{ind pairs} Suppose $K$ is a real closed field of transcendence degree at least $2$ and  every real closed subfield of positive transcendence degree is dense. Let $a,b\in K$ with $a<b$ and
  $f_0,g_0,f_1,g_1:(a,b)\rightarrow K$ be defined over $K$ such that $f_i,g_i$ are algebraically independent for $i=0,1$. Then there is $z\in (a,b)$
  such that $f_i(z),g_i(z)$ are algebraically independent for $i=0,1$.
  \end{Corollary}
  
 Note that at several points in the proof of Theorem \ref{main} we choose $x_i$ in a specified subfield of transcendence one--for example case 2 of the proof of Lemma \ref{both td 1}.  We need to insure that $x_i$ can also be chosen in $\img(f_i)$.  Our added assumption that all real closed subfields of positive transcendence degree are dense guarantees this.  While sufficient for our applications in \S 5, it would be interesting to have a proof of Corollary \ref{ind pairs} without this assumption.
  
  \section{Rigid real closed fields}\label{rigid}
  
  We say that a structure is {\em rigid} if it has no non-trivial automorphisms. In real closed fields, the positive elements are the nonzero squares. Thus
every automorphism of a real closed field preserves the ordering as well as the field structure.  If $K$ is an Archimedean real closed field, the field of rational numbers is dense in $K$ and fixed pointwise by all automorphisms; hence $K$ is rigid.  

Are there non-Archimedean rigid real closed fields?  In \cite{ss},  Shelah proved that it is consistent with ZFC that there are.  For example, he showed, assuming Jensen's combinatorial principle $\diamondsuit_{\omega_1}$,  that there are rigid non-Archimedean real closed fields of cardinality
$\aleph_1$.   In later work,   Mekler and Shelah \cite{ms} revisited Shelah's results and showed that  the existence of arbitrarily large rigid non-Archimedean real closed fields
could be proved in ZFC without extra set-theoretic assumptions.  Remarkably, they proved that the fragment of second order logic where all structures are ordered fields and we can quantify over their automorphism groups is compact.  Add to the language of ordered
fields a constant symbol $c$ and let $\Gamma$ be the real closed fields axioms together with the sentences
$${c>n: n\in \Nn}\cup\{\forall \sigma\forall x \ \sigma(x)=x\}.$$ For any finite $\Delta\subset\Gamma$ we can interpret $c$ in $\R$ so 
that $\Delta$ holds.  Thus, by compactness, there are arbitrarily large models of $\Gamma$.  They use Shelah's ``black box"
which is useful for doing some $\diamondsuit$-like constructions in ZFC.  The models built are large (of cardinality at least
$\beth_\omega$, say).

In 2018 Ali Enayat asked on MathOverflow if there are countable rigid non-Archimedean real closed fields.  In \cite{marker stein}
Marker and Steinhorn constructed a rigid real closed field of transcendence degree 2.  Our construction was later significantly generalized by Michael Lange \cite{lange} to build rigid closed fields of transcendence degree $n\ge 2$ for all $n\in \Nn$.  We thought that this would be the limit
of the method, but Lange found a clever diagonalization strategy to build a rigid non-Archimedean real closed field of transcendence degree $\aleph_0$.
 
 Here we introduce a different construction that allows us to build rigid non-Archimedean fields of transcendence degree $\kappa$
for all $2\le\kappa\le\aleph_1$. 

\begin{Definition}
  We say that a real closed field $K$ has {\em unique independent pairs} if whenever $a,b\in K$ are algebraically independent, then
  $(a,b)$ is the unique realization of $\tp(a,b)$ in $K$. 
  \end{Definition}
    
    \medskip
  If $K$ has unique independent pairs and  transcendence degree at least 2, then $K$ is rigid.  Also note that a rigid real closed field of transcendence degree 2 has unique independent pairs.   In particular, this is true of the field constructed in \cite{marker stein}.
  
  Real closed fields with unique independent pairs also satisfy the additional hypothesis of Corollary \ref{ind pairs}.
  
  \begin{Lemma} Suppose $K$ is a real closed field with unique independent pairs.  Then every real closed subfield of $K$ of positive transcendence degree is dense.
  \end{Lemma}
  \begin{Proof}. Suppose $k\subset K$ is a real closed subfield of positive transcendence degree that is not dense.  There is an interval $(a,b)$ in $K$ with $(a,b)\cap k=\emptyset$.  Let $y,z\in (a,b)$ with $y\ne z$.  Let $x\in k$ be transcendental.  By o-minimality, 
  $\tp(x,y)=\tp(x,z)$ and both pairs $x,y$ and $x,z$ are algebraically independent, contradicting unique independent pairs.
  \end{Proof}
  
  \medskip
Our goal is to prove the following theorem.  Here $K\< x\> $ denotes the real closure of $K(x)$.
  
  \begin{Theorem}\label{uip} Suppose $K$ is a countable  real closed field of transcendence degree at least 2 with unique independent pairs.  
  We can find a non-principal type $p\in S_1(K)$ such that $K\< x\> $ has unique independent pairs for any $x$ realizing $p$.
  \end{Theorem}
  
  Let $K$ be a non-Archimedean real closed field of transcendence degree at least 2 with unique independent pairs.
 We we will construct a type $p\in S_1(K)$ such that if $x$ realizes $p$, and $f_0,g_0,f_1,g_1:(a,b)\rightarrow K$ are $K$-definable
 functions then either:
 
 i) $(f_0(x),g_0(x))=(f_1(x), g_1(x))$;
 
 ii) $f_0(x)$ and $g_0(x)$ are algebraically dependent over $\Q^\rcl$;
 
 iii) $\tp(f_0(x),g_0(x)) \ne \tp(f_1(x),g_1(x))$
 
 \n Since every element of $K\< x\> $ is of the form $f(x)$ for some $K$-definable function, this will insure that $K\< x\> $
 has unique independent pairs.
 
 \medskip
 
 For $C$ an open cell, let $\cl(C)$ denote the topological closure of $C$.

  \begin{Lemma}\label{main lemma} Let $K$ be real closed field of transcendent degree at least 
  $2$ with unique independent pairs. Suppose  $a,b\in K$,  $a<b$ and $f_i,g_i:(a,b)\rightarrow K$ for $i=0,1$ are
  continuous monotonic, $K$-definable functions where $f_i$ and $g_i$ are algebraically independent on $(a,b)$ for $i=0,1$   and
  $f_0(x)\ne f_1(x)$ or $g_0(x)\ne g_1(x)$ for all $x\in (a,b)$.  
  There is an $\emptyset$-definable  $2$-cell $C$ defined over $\emptyset$ and an interval $(a^\prime,b^\prime)\subseteq (a,b)$
  such that
 $(f_i(x), g_i(x)))\in C$ and $(f_{1-i}(x),g_{1-i}(x))\not \in \cl(C)$ for any $x\in (a^\prime, b^\prime)$ for $i=0$ or 1.
 \end{Lemma}

  \begin{Proof}
  Let $C$ be cell.  Suppose there is $x\in (a,b)$ such that $(f_i(x),g_i(x))\in C$ and $(f_{1-i}(x),g_{1-i}(x))\not \in \cl(C)$ for $i=0$ or 1.
  By continuity we can find a subinterval $(a^\prime, b^\prime)$ such that our conclusion holds.  
  So suppose, for contradiction that $$(f_i(x),g_i(x))\in C\Rightarrow (f_{1-i}(x),g_{1-i}(x))\in\cl(C)$$ for $i=0$ and 1.
  
 By Corollary \ref{ind pairs}, we can choose $c\in (a,b)$ such that $f_i(c)$ and $g_i(c)$ are algebraically independent  for $i=0,1$.  Suppose $\phi(v_1,v_2)$ is a formula over  $\emptyset$ and $\phi(f_0(c), g_0(c))$.  Let $C$ be a cell defined over 
$\emptyset$ such that $(f_0(c), g_0(c))\in C$ and $C\subseteq \phi(K)$. Since $f_0$ and $g_0$ are algebraically independent,
$C$ must be an open cell.   We must have $(f_1(c),g_1(c))\in \cl(C)$.  But then,
by our choice of $c$, $f_1(c)$ and $g_1(c)$ are algebraically independent, and, hence, not on the boundary of $C$. Thus
$(f_1(c),g_1(c))$ lies in $C$.  Thus $\tp(f_0(c),g_0(c))=\tp(f_1(c),g_1(c))$ contradicting
the fact that $K$ has unique independent pairs.
  \end{Proof}

  \bigskip \n {\bf Proof of Theorem \ref{uip}}
  Let $K$ be a countable real closed field of transcendence degree at least 2 with unique independent pairs.  We construct a type $p\in S_1(K)$ in stages. 
  At stage $s$ we will have $a_s,b_s\in K$ with $a_s<b_s$. We will also have $K$-definable functions
 $f_0,f_1,g_0,g_1:(a_s,b_s)\rightarrow K.$  At this stage our goal is to shrink the interval to guarantee that if $x$ is a realization of the type then $(f_0(x), g_0(x))$ and $(f_1(x), g_1(x))$ are not two distinct pairs of algebraically independent elements realizing the same type over $\emptyset$.
 We do this by finding $(a_{s+1}^\prime,b_{s+1}^\prime)\subseteq (a_s,b_s)$
  such that either:
  
  i) $f_0(x)=f_1(x)$ and $g_0(x)=g_1(x)$ on $(a_{s+1}^\prime,b_{s+1}^\prime)$;
  
  ii) $f_0$ and $g_0$ are interalgebraic on $(a_{s+1}^\prime,b_{s+1}^\prime)$ or 
  $f_1$ and $g_1$ are interalgebraic on $(a_{s+1}^\prime,b_{s+1}^\prime)$;

  \n or
  
  iii) there is a formula $\phi(v_1,v_2)$ with no parameters such that 
 $$\phi(f_1(x),g_1(x))\land \neg\phi(f_2(x),g_2(x))\hbox{ for } x\in  (a_{s+1}^\prime,b_{s+1}^\prime).$$

  \medskip Without loss of generality, we may assume that $f_i$ and $g_i$ are continuous and 
  strictly monotonic on $(a,b)$.  If  there is a subinterval where i) holds we are done.  If not, we may, without loss of generality,
  assume that $f_0(x) \ne f_1(x)$ or $g_0(x)\ne g_1(x)$ for all $x\in (a_s,b_s)$.
  If there is a subinterval where $f_i$ and $g_i$ are interalgebraic, then ii) hold on that
  subinterval, So we may, without loss of generality, assume that neither pair is interalgebraic on a subinterval.
  By Lemma  \ref{main lemma}, we can find a subinterval $(a_{s+1},b_{s+1})$ and an $\emptyset$-definable cell $C$
  such that $$(f_i(x),g_i(x))\in C \land (f_{1-i}(x),g_{1-i}(x))\not\in \cl(C)$$ for all $x\in (a_{s+1},b_{s+1})$.
\hfill \qed

  \begin{Corollary}\label{stretch} There are non-Archimedean real closed fields with unique independent pairs of transcendence degree $\kappa$ for all $2\le\kappa\le \aleph_1$.  In particular, these fields are rigid.
  \end{Corollary}
  \begin{Proof} Let $K_0$ be the non-Archimedean rigid field of transcendence degree 2 constructed in \cite{marker stein}.  We can build
  $(K_\alpha:\alpha<\omega_1)$ such that $K_\alpha\subset K_{\alpha+1}$ for all $\alpha$,
  $K_\alpha=\bigcup_{\beta<\alpha} K_\beta$ for $\alpha$ a limit ordinal and each $K_\alpha$ is a
  real closed field with unique independent pairs.   Then $K=\bigcup_{\alpha<\omega_1}K_\alpha$ is a real closed field of 
  transcendence degree $\aleph_1$ with unique independent pairs.
  \end{Proof}
  
  \medskip Corollary \ref{stretch} is as far as we can go using Theorem \ref {uip} to build non-Archimedean real closed fields with unique independent pairs.
  It would be interesting to show, in ZFC, that there are non-Archimedean real closed fields with unique independent pairs of cardinality $2^{\aleph_0}$. 
  
 \medskip  Suppose $T$ is the theory of an o-minimal expansion of the real field in a countable language. 
 The constructions of \cite{marker stein} and \cite{lange} work for models of $T$. Replacing ``algebraic independence"
 by the usual notion of independence for o-minimal theories, we can ask if 
 the analogs of Corollary \ref{real main} and Theorem \ref{uip} hold for models of $T$?

  \end{document}